\documentclass{amsart}

\usepackage{amsmath}
\usepackage{amsfonts}
\usepackage{amssymb}
\usepackage{comment}
\usepackage{epsfig}
\usepackage{psfrag}
\usepackage{mathrsfs}
\usepackage{amscd}
\usepackage[all]{xy}
\usepackage{rotating}
\usepackage{lscape}
\usepackage{amsbsy}
\usepackage{verbatim}
\usepackage{moreverb}
\usepackage{epigraph}

\newtheorem{theorem}{Theorem}[section]
\newtheorem{prop}[theorem]{Proposition}

\newtheorem{cor}[theorem]{Corollary}

\theoremstyle{remark}
\newtheorem{remark}[theorem]{Remark}

\newtheorem{q}[theorem]{Question}

\DeclareMathOperator{\Hom}{Hom}
\DeclareMathOperator{\Fun}{Fun}

\DeclareMathOperator{\Map}{Map}

\DeclareMathOperator{\Comod}{Comod}
\DeclareMathOperator{\Mod}{Mod}
\DeclareMathOperator{\m}{Mod}

\DeclareMathOperator{\qc}{Q}

\def\1bord{1\mathrm{Bord}}
\def\2bord{2\mathrm{Bord}}
\def\3bord{3\mathrm{Bord}}

\def\oo{\infty}

\newcommand{\ra}{\rightarrow}

\def\cA{\mathcal A}\def\cC{\mathcal C}\def\cD{\mathcal D}
\def\cF{\mathcal F}\def\cH{\mathcal H}
\def\cL{\mathcal L}
\def\cM{\mathcal M}

\def\cZ{\mathcal Z}

\def\CC{\mathbb C}

\def\ot{\otimes}

\def\ot{\otimes}
\def\XYX{X \times_Y X}

\def\wt{\widetilde}

\newcommand{\on}{\operatorname}

\newcommand{\Ind}{\on{Ind}}

\newcommand{\ti}{\times}

\newcommand{\Id}{{\rm id}}

\newcommand{\aff}{\it{aff}}

\newcommand{\bs}{\backslash}

\def\dash{\textendash}

\def\perf{\on{Perf}}

\def\2Cat{{\rm Cat}_{(\infty, 2)}}

\def\End{\on{End}}

\begin{document}

\title{Morita equivalence\\ for convolution
  categories: Appendix to \cite{BFN}} \author{David Ben-Zvi, John Francis, and David Nadler}

\address{Department of Mathematics\\University of Texas\\Austin, TX
  78712-0257} \email{benzvi@math.utexas.edu} \address{Department of
  Mathematics\\Northwestern University\\Evanston, IL 60208-2370}
\email{jnkf@math.northwestern.edu} \address{Department of
  Mathematics\\Northwestern University\\Evanston, IL 60208-2370}
\email{nadler@math.northwestern.edu}

\maketitle

\begin{abstract}
In this brief postscript to \cite{BFN}, we describe a Morita equivalence for derived, categorified
matrix algebras implied by theory \cite{1affine,preygel,indcoh} developed since the appearance of \cite{BFN}.
We work in the setting of perfect stacks $X$ and their stable $\oo$-categories
 $\qc(X)$  
of quasicoherent sheaves.
Perfect stacks include all varieties
and common stacks in characteristic zero, and their stable $\oo$-categories of sheaves
are well behaved refinements of their quasicoherent derived categories, satisfying natural analogues
of common properties of function spaces.

To a morphism of perfect stacks $\pi:X\to Y$, we associate the categorified matrix algebra
$\qc(\XYX)$ 
of sheaves on the derived fiber product equipped with its monoidal convolution product.
We show that for $\pi$
faithfully flat (as a corollary of the 1-affineness theorem of
Gaisgory~\cite{1affine}) or for $\pi$ proper and surjective and $X,Y$ smooth (as an
application of proper descent \cite{preygel,indcoh}), there is a
Morita equivalence between $\qc(X\times_Y X)$ and $\qc(Y)$, that is,
an equivalence of their $\oo$-categories of stable module $\oo$-categories. In
particular, this immediately implies an identification of their Drinfeld centers (as previously established
in  \cite{BFN}), and more generally, an identification of their
associated topological field theories. Another consequence is that for
an affine algebraic group $G$ in characteristic zero and an algebraic subgroup
$K$, passing to $K$-invariants induces an equivalence from
stable $\oo$-categories with algebraic $G$-action to modules for the Hecke category
$\qc(K\bs G/K)$.
\end{abstract}


\section{Introduction}

In this postscript to \cite{BFN}, we establish a derived, categorified
analogue of the following well-known observation. Let $X\to Y$ denote a
surjective map of finite sets, and consider the vector spaces $\CC[X], \CC[Y]$ of functions. 
Pullback of functions realizes $\CC[X]$ as a module over $\CC[Y]$ equipped with its usual commutative pointwise multiplication.

Consider the vector space $\CC[X\times_Y X]$ of functions on the fiber product. Concretely, it  forms the matrix algebra of block diagonal matrices with
indices labelled by $X$ and blocks labelled by $Y$.
It naturally acts on $\CC[X]$ by matrix multiplication, and $\CC[Y]$ is the  subalgebra of block scalar matrices.
The $\CC[Y]\dash\CC[X\times_Y X]$ bimodule $\CC[X]$ provides a Morita equivalence identifying the categories of $\CC[Y]$-modules
and $\CC[X \times_Y X]$-modules.

We will
establish an analogue of this Morita equivalence in which finite sets
are replaced by perfect stacks (for example, varieties or common stacks in characteristic zero) and
 vector spaces of functions by stable $\oo$-categories of quasicoherent
sheaves.

\subsection{Statements of Morita equivalence}

We will continue in the general  setting of \cite{BFN}, and only comment here specifically where we interface with it. 
In particular,  we will work with perfect stacks $X$ as introduced in \cite{BFN}
where the stable $\oo$-category $\qc(X)$ of quasicoherent sheaves  is
``generated by finite objects" (see also the
discussion in  \cite[Section 8]{dag11}). 
More precisely, a derived stack $X$ with
affine diagonal is said to be perfect if  $\qc(X)$ 
is the inductive limit
$$ \qc(X) \simeq \Ind\perf(X)
$$ of the full $\oo$-subcategory $\perf(X)$ of perfect complexes (objects locally equivalent to finite complexes of finite
vector bundles), or
equivalently, $\qc(X)$ is compactly generated (there is no right orthogonal to the compact objects), and compact
objects and dualizable objects coincide.

The class of
perfect stacks is very broad, and includes all (quasi-compact and
separated) schemes and common stacks in characteristic zero (in
particular, the quotient $X/G$ of a quasi-projective derived scheme $X$
by a linear action of an affine group $G$). It is also closed under
fiber products, passing to total spaces of quasiprojective morphisms,
and quotients by finite group schemes in characteristic zero.

The main technical result of \cite{BFN} is the following theorem.

\begin{theorem}[\cite{BFN}]\label{main BFN}
Let $\pi:X\to Y$ denote a morphism between perfect stacks. Then
$\qc(X)$ is canonically self-dual as a $\qc(Y)$-module, and there are
canonical equivalences of $\qc(X)$-bimodules
 $$
 \xymatrix{ \qc(X)\ot_{\qc(Y)} \qc(X) \ar[r]^-\sim & 
\qc(X\times_Y X) 
 \ar[r]^-\sim & 
\Hom_{\qc(Y)}
(\qc(X) , \qc(X))
}
 $$ 
  
Moreover,  the convolution monoidal structure on the middle
term is identified with the composition monoidal structure of the third term.
 \end{theorem}

\begin{remark}
For perfect stacks $X_1, X_2$ and a morphism $f:X_1\to X_2$  over the perfect stack $Y$,
the canonical self-duality of $\qc(X_1), \qc(X_2)$ over $\qc(Y)$ identifies the transpose of the pullback $f^*:\qc(X_2)\to \qc(X_1)$
with the pushforward $f_*:\qc(X_1)\to \qc(X_2)$ and vice versa. Thus not only is the
natural algebra structure on $\qc(X\times_Y X) $ compatible with that on 
$\Hom_{\qc(Y)}
(\qc(X) , \qc(X))$,
but also the  natural coalgebra structure on $\qc(X\times_Y X) $ is compatible with that on $\qc(X)\ot_{\qc(Y)} \qc(X)$.
\end{remark}

We will focus here on the module theory of the convolution
algebra $\qc(X\times_Y X)$, and its description in terms of the commutative algebra $\qc(Y)$. 
In general, when we speak about modules for an algebra object $\cA$ in a symmetric monoidal $\oo$-category $\cC$,
we mean $\cA$-module objects in $\cC$ and denote their $\oo$-category by $\Mod(\cA)$.  In particular, we regard the monoidal $\oo$-categories $\qc(Y)$,  $\qc(X\times_Y X)$ as algebra objects in the symmetric monoidal $\oo$-category $St$ of stable presentable $\oo$-categories.  

In \cite{BFN}, we constructed (under the hypothesis that
descent holds for $\pi$) an equivalence of  Drinfeld centers (or
Hochschild cohomology categories) 
$$
\xymatrix{
\cZ(\qc(Y)) \simeq \qc(\cL Y)\simeq \cZ(\qc(X\ti_Y X)).
}$$
where $\cL Y = \Map(S^1, Y)$ is the derived loop space of $Y$. 
Under appropriate assumptions, we will upgrade this here to a Morita equivalence between the two algebras.

\begin{theorem}\label{main theorem}
Let $\pi:X\to Y$ denote a morphism of perfect stacks.  If either
\begin{enumerate}
\item $\pi$ is faithfully flat, or 
\item $X,Y$ are smooth and $\pi$ is proper and surjective,
\end{enumerate}
then the $\qc(Y)\dash\qc(X\times_Y X)$-bimodule $\qc(X)$ defines an equivalence
$$
\xymatrix{
\qc(X) \otimes_{\qc(Y)} (-):\Mod(\qc(Y)) \ar[r]^-\sim &  \Mod(\qc(X\times_Y X))
}$$ 
of $\oo$-categories of stable presentable module categories.
\end{theorem}

The proof is a straightforward application of Lurie's Barr-Beck Theorem.
An evident necessary condition is that  the functor
$\qc(X) \otimes_{\qc(Y)} (-)$ is conservative.
The stated assumptions provide common geometric settings where this can be seen by appealing to established theory.

First, the flat Morita
equivalence is a categorified analogue of faithfully flat descent for
quasicoherent sheaves (due to Lurie \cite{dag8,dag11} in the derived
setting). Here the key input is the theorem of Gaitsgory
\cite{1affine} that many geometric stacks $X$ (including perfect
stacks) are {\em 1-affine} in the sense that localization and global sections
define inverse equivalences between $\qc(X)$-modules and quasicoherent
stacks of categories over $X$.

Second, the strange hybrid setting for the proper Morita equivalence results from combining two inputs. 
On the one hand, there is
a general descent theorem for proper morphisms due to Preygel and
Gaitsgory (see \cite[Proposition A.2.8]{preygel} and
\cite[7.2.2]{indcoh}) which applies not to $\qc(X)$ but to its ``dual counterpart"
 the $\oo$-category $\qc^!(X)=\Ind(\on{Coh}(X))$ of {\em
  ind-coherent sheaves}. Just as quasicoherent sheaves are a categorified analogue of
functions or cohomology, ind-coherent sheaves are a categorified analogue of
measures or homology (many of their beautiful properties
are developed in \cite{indcoh}). However, the analogue of Theorem
\ref{main BFN} {\em does not hold} for ind-coherent sheaves even in
the restricted setting of Theorem \ref{main theorem}(2) (though it
holds in the ``absolute'' case when $Y=pt$), and proper Morita equivalence for $\qc^!$ also fails.

On the other hand, quasicoherent and ind-coherent sheaves coincide if and only if the geometric stack is smooth.
Thus our assumption that $X,Y$ are smooth  implies that we have proper
descent for quasicoherent sheaves. Moreover, while usual
descent data
is given by a coalgebra (or comonad), proper descent uses an opposite
adjunction and hence is described by an algebra (monad). Thus we can deduce proper Morita equivalence 
from general facts about modules over algebras.

\subsection{Applications}

First, we mention some relations to topological field theory.  Recall
\cite{jacob TFT} the Morita $(\oo, 2)$-category $2Alg$ of algebra
objects in stable presentable $\oo$-categories $St$ with
$1$-morphisms given by bimodules, and $2$-morphisms given by natural
transformations.

\begin{cor} 
Let $\pi:X\to Y$ satisfy either hypotheses (1) or (2) of
Theorem \ref{main theorem}, and $\cH=\qc(X \times_Y X)$ denote the resulting convolution
 category.

\begin{enumerate}
\item  As an object of
  $2Alg_k$,  the convolution category $\cH$ is equivalent to $\qc(Y)$, and hence it is 2-dualizable and admits a canonical $O(2)$-fixed
  structure.

\item The  two-dimensional unoriented topological field
  theory corresponding to $\cH$ is equivalent to that corresponding to $\qc(Y)$.

\item The Hochschild homology and cohomology of $\cH$ are canonically
  identified as 
  $$HH_*(\cH)\simeq HH^*(\cH)\simeq \qc(\cL Y).
  $$

\end{enumerate}
\end{cor}

\begin{proof} Theorem \ref{main theorem}
provides an equivalence between $\cH$ and $\qc(Y)$ as objects of
$2Alg$, so the remaining assertions follow from the corresponding assertions
for $\qc(Y)$ itself. The 2-dualizability of $\qc(Y)$ follows from the
commutativity, or can be seen from the more general criterion for
$SO(2)$-fixed monoidal categories of \cite{character}. In particular,
the canonical $SO(2)$-invariant trace is given by the right adjoint to
the unit morphism $\cF\mapsto \Hom(1,\cF)$.
\end{proof}

Now let 
$G$ be an affine algebraic group, and
 $X = pt$ a point,
and $Y=BG$ the classifying stack of $G$, so that then $G = pt \times_{BG} pt$.

Define the {\em
  quasicoherent group algebra} of $G$ to be the $\oo$-category
$\qc(G)$ equipped with its monoidal structure coming from convolution along
the group multiplication, or equivalently, the natural convolution product on $\qc(pt \times_{BG} pt)$.
Define a {\em
  quasicoherent $G$-category} to a $\qc(G)$-module
category. 

As a specific instance of Gaitsgory's general results~\cite{1affine}, the classifying stack
$BG$ is {\em 1-affine}: 
 localization and global sections
define inverse equivalences between
$\qc(BG)$-modules and quasicoherent stacks of
categories over $BG$. As a consequence, 
taking $\qc(G)$-invariants gives an equivalence from  quasicoherent $G$-categories
to  $\qc(BG)$-modules.  This is a categorified analogue
of the Koszul duality theorem of Goresky-Kottwitz-MacPherson
\cite{GKM} that provides a derived Morita equivalence (under certain
auxiliary conditions) between the ``topological group algebra"
$C_*(G)$ and equivariant cochain complex $C^*(BG)$. Note this is in
marked distinction with ordinary representation theory of a finite or
reductive group, where passage to $G$-invariants is essentially never
an equivalence.

The flat Morita equivalence (itself an application of ~\cite{1affine}) gives the following generalization. For an algebraic subgroup $K\subset
G$, consider the double coset stack $K\bs G/K=BK\times_{BG} BK$
and the resulting Hecke algebra $\qc(K\bs
G/K)$. By Theorem~\ref{main theorem}, such algebras are all
Morita equivalent to $\qc(BG)$, and hence to each other. Tracing through
the constructions, the Morita functor from quasicoherent $G$-categories to
Hecke modules is given by passage to $K$-equivariant objects.

\begin{cor} 
For affine algebraic groups $K\subset G$, there is a canonical equivalence 
$$
\xymatrix{
(-)^K:\Mod(\qc(G))\ar[r]^-\sim &  \Mod(\qc(K\bs G/K))
}$$ which sends a
quasicoherent $G$-category $\cM$ to the $\oo$-category of
$K$-equivariant objects $$
\xymatrix{
\cM^K\simeq \Hom_{\qc(G)}(\qc(G/K),\cM)
}$$
\end{cor}

The special setting (a proper map between smooth stacks) of the proper
Morita equivalence has an important instance: the
Grothendieck-Springer simultaneous resolution 
$$\xymatrix{
\mu:\wt{G}/G\ar[r] &  G/G
}$$ of the adjoint quotient of a 
reductive group. The corresponding fiber product is the adjoint quotient of the Steinberg variety
$${\mathcal St}/G = \wt G/G \times_{G/G} \wt G/G \simeq (\wt G \times_G \wt G)/G$$

\begin{cor} There is a canonical equivalence
$$
\Mod(\qc({\mathcal St}/G) )\simeq \Mod(\qc(G/G))
$$ 
of objects of $2Alg$
The corresponding 
unoriented two-dimensional topological field theory valued in $2Alg$
assigns to the circle the stable $\oo$-category
$$\qc(\cL (G/G))=\qc(BG^{T^2})=\qc({\mathcal Loc}_G(T^2))$$ 
of quasicoherent sheaves on the
derived stack of $G$-local systems on the two-torus $T^2$.
\end{cor}

\begin{remark}
A primary motivation for the consideration of $\qc({\mathcal St}/G)$
comes from its close relation to the affine Hecke category  $\cH^{\aff}$ refining the affine Hecke algebra.
Namely, 
 a theorem
of Bezrukavnikov  \cite{Roma ICM} describes  $\cH^{\aff}$ (or its ind-complete version) in terms of $G$-equivariant coherent (or ind-coherent)
sheaves on the Steinberg stack ${\mathcal St}/G$. 

From this perspective, the study of $\cH^{\aff}$ is subtle 
in part due to the failure of Morita equivalence: in fact, the natural action of $\cH^{\aff}$ on the standard module $\qc(\wt G/G)$
factors through the endofunctors $\qc({\mathcal St}/G)$.
The Morita equivalence of the above corollary can be viewed as an
 incomplete categorification of the
Kazhdan-Lusztig description of modules over the affine Hecke algebra.
It successfully describes $\cH^{\aff}$-modules accessible via the standard module $\qc(\wt G/G)$.

A central problem going forward is to extend this picture to the full affine Hecke
category $\cH^{\aff}$. For example, one might expect the Hochschild homology of $\cH^{\aff}$ to involve those ind-coherent sheaves with nilpotent singular support  on the derived stack of 
$G$-local systems on the two-torus $T^2$
(as appears in the refined geometric
Langlands conjecture of \cite{AG}).

\end{remark}

\subsection{Acknowledgements} 
We would like to thank Andrew Blumberg, Dennis Gaitsgory and Jacob
Lurie for helpful remarks. We would also like to thank Dennis Gaitsgory for 
sharing the preprint \cite{1affine} with us.


\section{Morita equivalence via Barr-Beck Theorem}

Let $\pi:X\to Y$ be a morphism of perfect stacks.

We have the standard adjunction
$$
\xymatrix{
\pi^*:\qc(Y)\ar@{<->}[r] &  \qc(X):\pi_*
}
$$ 
The pullback $\pi^*$ is monoidal, and
$\qc(X)$ is canonically self-dual as a $\qc(Y)$-module (as stated in Theorem~\ref{main BFN} and proved in \cite{BFN}).  

Consider
the induced adjunction
$$\xymatrix{
\Pi^*:\Mod(\qc(Y))\ar@{<->}[r] & \Mod(\qc(X)):\Pi_*
}
$$ where the
pullback $\Pi^*$ is given by 
$$
\xymatrix{
\Pi^*(-) = \qc(X)\ot_{\qc(Y)}(-)
}
$$
and the pushforward $\Pi_*$ is the forgetful functor given by restricting
$\qc(X)$-modules to $\qc(Y)$-modules along $\pi^*$.  

By construction, the pullback $\Pi^*$ preserves colimits, and since $\qc(X)$ is dualizable over $\qc(Y)$, it also preserves limits.

Now consider the convolution algebra 
$$
\xymatrix{
\qc(X \times_Y X) \simeq \End_{\qc(Y)}(\qc(X))
}$$

\begin{prop}\label{main taut}
Assume $\Pi^*$ is conservative.

Then the functor
  $$
  \xymatrix{
  \qc(X)\ot_{\qc(Y)}(-): \Mod(\qc(Y))\ar[r] & 
  \Mod(\qc(X \times_Y X) )
  }
  $$
  is an equivalence.
  \end{prop}

\begin{proof}
Applying the comonadic form of Lurie's Barr-Beck Theorem to the adjunction
$$\xymatrix{
\Pi^*:\Mod(\qc(Y))\ar@{<->}[r] & \Mod(\qc(X)):\Pi_*
}
$$
we have an equivalence
$$
\xymatrix{
  \qc(X)\ot_{\qc(Y)}(-): \Mod(\qc(Y)) \ar[r]^-\sim & \Comod_{T^\vee}(\Mod(\qc(X)))
}
$$
where $T^\vee$ is the comonad of the adjunction with underlying endofunctor $\Pi^* \Pi_*$.

To calculate $T^\vee$, observe that it is represented by the coalgebra  object
$$
\xymatrix{
T^\vee(\qc(X)) \simeq \qc(X) \otimes_{\qc(Y)} \qc(X)
}
$$ 
inside of $\qc(X)$-bimodules.
The fact that  $\qc(X)$ is  dualizable over $\qc(Y)$ implies that 
$$
\xymatrix{
\qc(X) \otimes_{\qc(Y)} \qc(X)
\simeq
\qc(X \times_Y X)
}
$$
is  dualizable as a $\qc(X)$-bimodule. Moreover,  the canonical self-duality of $\qc(X)$  over $\qc(Y)$ provides
a canonical identification  of the $\qc(X)$-bimodule dual of  this coalgebra with the algebra object
$$
\xymatrix{
\End_{\qc(Y)} \qc(X) \simeq \qc(X \times_Y X)
}
$$
\end{proof}


Now we present two geometric contexts where the conservativity assumed in Proposition~\ref{main taut} can be readily verified.

\subsection{Flat Morita equivalence for perfect stacks}

\begin{prop}
Suppose $\pi:X\to Y$ is faithfully flat. Then $\Pi^*$ is conservative.
\end{prop}

\begin{proof}
Observe that
$\Pi^*$ factors as a composition of the localization functor from $\qc(Y)$-modules to quasicoherent stacks on
$Y$, pullback of quasicoherent stacks along $\pi$, and the global
sections functor from quasicoherent stacks on $X$ to
$\qc(X)$-modules.

By the 1-affineness of $X$ and $Y$ established in \cite{1affine}, the first and third functors are
equivalences. 
Finally, the pullback of quasicoherent stacks along $\pi$ is
conservative since the map $\pi$ has a section locally in the flat
topology whose pullback provides a left inverse to pullback along $\pi$.
\end{proof}

\subsection{Proper Morita equivalence for smooth perfect stacks}

Now assume $\pi:X\to Y$ is a proper surjective morphism of perfect stacks.  

We have the adjunction
$$
\xymatrix{
\pi_*: \qc^!(X) \ar@{<->}[r] &  \qc^!(Y) : \pi^!
}
$$
Consider the induced monad $T$ with underlying endofunctor $\pi^!\pi_*$.
By the proper descent of \cite{preygel,indcoh},
we have an equivalence 
$$
\xymatrix{
\qc^!(Y)\simeq \Mod_T(\qc^!(X))
}$$

Now assume in addition that $X$ and $Y$ are smooth so that we have equivalences 
$$
\xymatrix{
\qc(X)\simeq \qc^!(X) &  \qc(Y)\simeq \qc^!(Y)
}
$$
Thus 
we have an equivalence 
$$
\xymatrix{
\qc(Y)\simeq \Mod_T(\qc(X))
}$$
where the monad $T$ is represented by an algebra object in $\End_{\qc(Y)}(\qc(X))$.

\begin{prop}
Suppose $X$ and $Y$ are smooth and $\pi:X\to Y$ is proper and surjective. Then $\Pi^*$ is conservative.
\end{prop}

\begin{proof}
We appeal to Proposition~\ref{algebras and monads} below. Namely, we can produce a left inverse to the functor $\Pi^*$ by passing to $T$-module objects
$$
\xymatrix{
\Mod_T(\Pi^*(\cM)) \simeq \Mod_T(\qc(X)) \otimes_{\qc(Y)} \cM \simeq \cM
}$$
\end{proof}

\section{Algebras and Module Categories}
In this section, we describe some useful relations between categories of
modules over algebras in a monoidal category $\cC$ and $\cC$-module
categories. The results form a straightforward modification of
\cite[Proposition 4.1]{BFN} from symmetric monoidal to general
monoidal $\oo$-categories $\cC$. The proofs are included for the
convenience of the reader.

\begin{prop}\label{algebras and monads}
Let $\cC$ be a stable presentable monoidal $\infty$-category, and
$A\in \cC$ an associative algebra object.  For any $\cC$-module $\cM$,
there is a canonical equivalence of $\infty$-categories
$$
{}_A\cC \ot_\cC \cM \simeq {}_A\cM.
$$

\end{prop}

\def\m{\mathrm{Mod}}


\begin{proof}


We will prove that ${}_A\cC\ot_\cC \cM$ is equivalent to
${}_A\cM$ by the natural evaluation functor.  Consider the
adjunction
$$
\xymatrix{\cC\ar@/^1pc/[r]^-{F}&{}_{A}\cC \ar[l]^-{G}
}
$$
where $F(-) =A\ot -$ is the induction, and $G$ is the forgetful functor.

The above adjunction induces an adjunction
$$
\xymatrix{\cM\ar@/^1pc/[r]^-{F\ot \Id}&{}_A\cC \ot_\cC \cM \ar[l]^-{G\ot\Id}
\ar[r] &\m_T(\cM)
}
$$ and thus a functor to modules over the monad $T=(G\ot{\rm
  id})\circ(F\ot{\rm id})$ acting on $\cM$.  The functor underlying
$T$ is given by tensoring with $A$, so we also have an equivalence
$\m_T(\cM)\simeq {}_A\cM$.

By its universal characterization, the functor $G'=G\ot {\rm id}$ is
colimit preserving.  We will now check that $G'$ is
conservative. It will follow that $G'$ satisfies the monadic
Barr-Beck conditions, and we obtain the desired equivalence ${}_A\cC
\ot_\cC \cM\simeq {}_A\cM$.

Observe that for any $\cD$, the pullback
$$
\Fun^{\rm L}({}_A\cC \ot_\cC \cM, \cD) \ra \Fun^{\rm L}(\cM, \cD)
$$ induced by the induction $F: \cM \to {}_A\cC \ot_\cC \cM$ is
conservative. In other words, if a functor out of ${}_A\cC \times \cM$
(which preserves colimits in each variable) is trivial when restricted
to $\cM$, then it is necessarily trivial.

Consequently, switching to opposite categories, we have that the corresponding functor
$$\Fun^{\rm R} (\cD, {}_A\cC \ot \cM) \ra \Fun^{\rm R} (\cD, \cM)
$$ induced by the forgetful functor $G': {}_A\cC \ot_\cC \cM \to \cM$
is conservative. Now we can apply \cite[Lemma 4.2]{BFN}, which asserts
that we can check conservativity of a colimit preserving right adjoint
$G$ on right adjoint functors from test categories $\cD$ (in fact the
lemma applies in a setting that is linear over some stable presentable
symmetric monoidal $\oo$-category).  This concludes the proof.

\end{proof}

\begin{prop}\label{algebra of algebras}
Let $\cC$ be a stable presentable monoidal $\infty$-category.

\begin{enumerate}

\item
For $A,A'\in \cC$ associative algebra objects, there is a canonical
equivalence of $\infty$-categories
$$ {}_A\cC_{A'} \simeq {}_A\cC\otimes_\cC \cC_{A'}.
$$

\item The $\infty$-category of modules ${}_A\cC$ is dualizable as a
  $\cC$-module with dual given by the $\infty$-category of modules
  $\cC_A$ over the opposite algebra. In particular for $A,A'$ algebra
  objects we have an equivalence between functors and bimodules
  $Hom_{\cC}(\cC_A,\cC_{A'})\simeq {}_A\cC_{A'}$.

\end{enumerate}

\end{prop}

\begin{proof}
We apply Proposition \ref{algebras and monads} to the instance
where $\cM$ is the $\infty$-category of left modules over another
associative algebra $A'$ to conclude that there is a natural
equivalence ${}_A\cC\ot_\cC \cC_{A'} \simeq {}_A(\cC_{A'})$. We now
have a chain of adjunctions
$$ \xymatrix{\cC\ar@/^1pc/[r]^-{F'}&\cC_{A'}\ar@/^1pc/[r]^-{F''}
  \ar[l]^-{G'}&{}_A(\cC_{A'})\ar[l]^-{G''}\\}
$$ in which the composite $G'\circ G''$ is colimit preserving and
conservative, and hence satisfies the monadic Barr-Beck conditions.

Furthermore, the above adjunction naturally extends to a diagram in which the cycle of left adjoints (denoted
by bowed arrows), and hence
also the cycle of right adjoints (denoted by straight arrows),  commute
$$
\xymatrix{\ar@/_1pc/[dr]_-{F} \cC\ar@/^1pc/[rr]^-{F''F'} &&  {}_A(\cC_{A'})\ar[ll]^-{G'G''}
\ar@/^1pc/[dl]^-{f} 
\\
& \ar[ul]_-G {}_A\cC_{A'} \ar[ur]^-{g} &
}
$$
Here $F(-) =A_1\ot A_2 \ot -$ is the induction, $G$ is the forgetful functor,
$f$ is the natural functor factoring through ${}_A\cC\ot_\cC\cC_{A'}$,
and $g$ is its right adjoint. From this diagram, we obtain a morphism of monads
$$
\xymatrix{
G'G''F''F' \ar[r] &  G'G''gfF''F' \simeq GF.
}
$$

Now the underlying functors of the monads $GF(-)$ and $G'G''F''F'(-)$ are both equivalent to the tensor $A\ot (-) \ot A'$, so 
the above morphism of monads is an equivalence. Thus we obtain the promised equivalence
${}_A\cC\ot_\cC\cC_{A'} \simeq {}_A(\cC_{A'})\simeq {}_A\cC_{A'}$. 

Finally, we show that the $\infty$-category of left $A$-modules
${}_A\cC$ is a dualizable $\cC$-module by directly exhibiting the
$\infty$-category of right $A$-modules $\cC_{A}$ as its
dual. The trace map is given by the two-sided bar construction
$$
\tau:  \cC_A \ot_\cC
{}_A\cC \to \cC
\qquad
M, N \mapsto M\ot_A N
$$ 
The unit map  is given by the induction
$$ u:\m_k \to {}_A\cC \ot_\cC \cC_A \simeq {}_A\cC_A \qquad V \mapsto
A\ot V
$$
where we regard $A\ot c$ as an $A$-bimodule.

One can verify directly that the composition
$$
\xymatrix{\m_A(\cC)\ar[r]^-{{\rm id}\ot u}&\m_A(\cC)\ot_\cC \m_{A^{\rm op}}(\cC)\ot_\cC\m_A(\cC)
\ar[r]^-{\tau\ot{\rm id}} &\m_A(\cC)}
$$ 
is equivalent to the identity. First, $({\rm id}\ot u)(M)$ is equivalent to $A \ot M$ 
regarded as an $A\ot A^{\rm op} \ot A$-module, and
second, $(\tau\ot{\rm id})(A\ot M)$ is equivalent to $A\ot_A M \simeq M$.
\end{proof}


\end{document}